\documentclass[12pt]{article}

\usepackage{amsmath, amsthm, amssymb}

\usepackage{fullpage}

\usepackage{hyperref}

\newtheorem{theorem}{Theorem}
\newtheorem{lemma}{Lemma}

\newtheorem{corollary}{Corollary}
\theoremstyle{definition}

\newtheorem{definition}{Definition}
\theoremstyle{remark}

\newcommand{\cF}{\mathcal{F}}

\title{Set systems containing no singleton intersection and the Delsarte number}
\author{William Linz\thanks{University of South Carolina, Columbia, SC, USA. ({\tt wlinz@mailbox.sc.edu}). Partially supported by NSF RTG Grant DMS 2038080.}}
\date{\today}

\begin{document}
\maketitle

 \begin{abstract}
 \noindent
 We prove that the maximum size of a family of $k$-element subsets of the set $[n] = \{1, 2, \ldots, n\}$ which contains no singleton intersection is $\binom{n-2}{k-2}$ when $3k-3 \le n \le k^2-k+1$. This improves upon a recent result of Cherkashin. Our proof uses Schrijver's variant of the Lov\'asz number and furnishes an infinite family of graphs where the Schrijver variant of the Lov\'asz number is strictly smaller than the Lov\'asz number. As a consequence of our result and a recent result of Keller and Lifshitz, it follows that for $k$ sufficiently large, the maximum size of a $k$-uniform family on $[n]$ containing no singleton intersection is $\binom{n-2}{k-2}$ for all $n\ge 3k-3$, which is the best possible threshold.
 \end{abstract}

\section{Introduction}
For positive integers $n$ and $k$, we denote the family of $k$-element subsets of $[n] := \{1, 2, \ldots, n\}$ by $\binom{[n]}{k}$. A family of sets $\cF \subset \binom{[n]}{k}$ is \emph{$t$-intersecting} if for all distinct $F, F'\in \cF$, it holds that $|F\cap F'| \ge t$. In the case that $t=1$, the family is called \emph{intersecting}. The classical Erd\H{o}s-Ko-Rado theorem~\cite{Erdos1961} shows that the maximum size of an intersecting family is $\binom{n-1}{k-1}$ when $n\ge 2k$ and that the maximum size of a $t$-intersecting family is $\binom{n-t}{k-t}$ when $n$ is sufficiently large. 

There has been much work done both on refining the original results of Erd\H{o}s, Ko and Rado for intersecting families and on extending the results to new settings. Many questions in extremal combinatorics which are extensions of the Erd\H{o}s-Ko-Rado theorem can be phrased as determining the maximum size of a $(n, k, L)$-system. 

\begin{definition}[$(n, k, L)$-system] 
Let $n$ and $k$ be positive integers and let $L$ be a set of integers such that $L \subset [0, k-1]$. Then a family of sets $\cF \subset \binom{[n]}{k}$ is an \emph{$(n, k, L)$-system} if for all distinct $F, F'\in \cF$, it holds that $|F\cap F'| \in L$. 
\end{definition}

It is often convenient to study independent sets of the \emph{generalized Johnson graph} $G(n, k, L)$ instead of studying $(n, k, L)$-systems directly. 

\begin{definition}[Generalized Johnson graph]
Let $n$ and $k$ be positive integers and let $L$ be a set of integers such that $L \subset [0, k-1]$. The \emph{generalized Johnson graph} $G(n, k, L)$ is the graph with $V(G(n, k, L)) = \binom{[n]}{k}$ and $AB\in E(G(n, k, L)) \iff |A\cap B| \notin L$. 

If $L = \{0, 1, \ldots, k-1\} \setminus \{\ell\}$, then we write $G(n, k, \ell)$ in place of $G(n, k, \{0, 1, \ldots, k-1\} \setminus \{\ell\})$. 
\end{definition}

Note that by definition the independent sets in the graph $G(n, k, L)$ correspond to $(n, k, L)$-systems, so the maximum size of an $(n, k, L)$-system is equal to the independence number $\alpha(G(n, k, L))$. 

After the $(t$-)intersecting families, the $(n, k, L)$-systems which have been most studied are those in which $L$ misses only one intersection size. Answering a question of Erd\H{o}s and S\'os~\cite{Erdos1975}, Frankl~\cite{Frankl1977} determined the maximum size of a family of sets containing no singleton intersection for sufficiently large $n$. 

\begin{theorem}[Frankl]\label{thm:franklsingleton} 
Let $k\ge 4$. Let $\cF \subset \binom{[n]}{k}$ be a $(n, k, \{0, 2, \ldots, k-1\})$-system. Then there exists an $n_0(k)$ such that for $n\ge n_0(k)$, 
\[|\cF| \le \binom{n-2}{k-2}.\]
\end{theorem}

Frankl and F\"uredi~\cite{Frankl1985} would later prove similar results for $L = \{0, \ldots, k-1\} \setminus \{t-1\}$; in particular, they determined the order of magnitude of $\alpha(G(n, k, t-1))$ for all $k$ and $t$. By now the maximum size of a set system missing exactly one intersection size is known for almost all choices of the parameters $n, k, t$ provided that $n$ or $k$ is sufficiently large~\cite{Ellis2024, Keller2021, Kupavskii2024}. 

On the other hand, there are very few exact results on set families with a forbidden intersection size which hold for small values of $n$ and all values of $k$. Keevash, Mubayi and Wilson~\cite{Keevash2006} determined the maximum size of a $4$-uniform family with no singleton intersection for all values of $n$. Very recently, Cherkashin~\cite{Cherkashin2024} proved the following result by using the Hoffman bound. 

\begin{theorem}[Cherkashin]\label{thm:cherkashin}
Let $k > 1$. Suppose that $\cF$ is a $(k^2 - k + 1, k, \{0, 2, 3, \ldots, k-1\})$-system. Then
\[|\cF| \le \binom{k^2-k-1}{k-2}.\]
\end{theorem}

In this note, we improve on Cherkashin's result by determining the maximum size of a $(n, k, \{0, 2, \ldots, k-1\})$-system when $3k-3 \le n \le k^2-k+1$. 

\begin{theorem}\label{thm:nosingletonint}
Let $n$ and $k$ be integers with $k\ge 3$ and $3k-3\le  n \le k^2-k+1$. Suppose that $\cF \subset \binom{[n]}{k}$ is a $(n, k, \{0, 2, 3, \ldots, k-1\})$-system. Then 
\[|\cF| \le \binom{n-2}{k-2}.\]
\end{theorem}

The lower bound on $n$ in Theorem~\ref{thm:nosingletonint} is best possible, as there are $2$-intersecting families with size greater than $\binom{n-2}{k-2}$ when $n < 3k-3$ (see \cite{Ahlswede1997}). The proof of Theorem~\ref{thm:nosingletonint} follows from a computation of Schrijver's variant of the Lov\'asz number for the graph $G(n, k, 1)$. In the course of the proof, we show that the graphs $G(n, k, 1)$ give an infinite family of examples for which Schrijver's variant of the Lov\'asz number is strictly smaller than the Lov\'asz number. 

Keller and Lifshitz~\cite[Theorem 1.4]{Keller2021} proved a very general result on Tur\'an-type problems where the extremal family is the $t$-star. We state the implication of their result for families containing no singleton intersection. 

\begin{theorem}[Keller--Lifshitz]\label{thm:kellerlifshitz}
There exists a constant $C$ such that the following holds. Let $C\le k \le n/C$ and let $\cF\subset \binom{[n]}{k}$ be a $(n, k, \{0, 2, 3, \ldots, k-1\})$-system. Then $|\cF| \le \binom{n-2}{k-2}$. 
\end{theorem}

The following corollary immediately follows from Theorem~\ref{thm:nosingletonint} and Theorem~\ref{thm:kellerlifshitz}.

\begin{corollary}
Let $C$ be a sufficiently large integer. Let $\cF \subset \binom{[n]}{k}$ be a $(n, k, \{0, 2, 3, \ldots, k-1\})$-system. For all $k\ge C$ and $n\ge 3k-3$,
\[|\cF| \le \binom{n-2}{k-2}.\]
\end{corollary}

Thus, for sufficiently large $k$ the best possible threshold $n_0(k) = 3k-3$ for a family $\cF$ with no singleton intersection to have size at most $\binom{n-2}{k-2}$ is the same as the threshold for a $2$-intersecting family to have size at most $\binom{n-2}{k-2}$.

\section{The Delsarte number and the proof of Theorem~\ref{thm:nosingletonint}} 

For any graph $G = (V, E)$, the Lov\'asz number of $G$ is defined as follows: for any matrix $A$, let $\text{lev}(A)$ denote the largest eigenvalue of $A$. Then, 
\[\vartheta(G) = \min\{\text{lev}(A): \text{$A$ is a symmetric $|V|\times |V|$ matrix such that $a_{ij} = 1$ if $\{i, j\}\notin E$}\}.\]
The Lov\'asz number provides an upper bound for the Shannon capacity $\Theta$. 

\begin{theorem}[Sandwich theorem\cite{Lovasz1979}]\label{thm:sandwichthm}
For any graph $G$,
\[\alpha(G) \le \Theta(G) \le \vartheta(G).\]
\end{theorem}

Schrijver~\cite{Schrijver1979} proposed the following strengthening of the Lov\'asz number, based on the Delsarte~\cite{Delsarte1973} linear programming bound. 
\[\vartheta'(G) = \min\{\text{lev}(A): \text{$A$ is a symmetric $|V|\times |V|$ matrix such that $a_{ij} \ge 1$ if $\{i, j\}\notin E$}\}.\]

For graphs which are unions of classes of graphs in an association scheme, Schrijver's variant of the Lov\'asz number is equal to the Delsarte linear programming bound. Since the graphs we consider in this note are unions of classes of graphs in the Johnson scheme, we call $\vartheta'$ the \emph{Delsarte number}. Schrijver~\cite[Theorem 1]{Schrijver1979} proved that the Delsarte number is sandwiched between the independence number and the Lov\'asz number. 

\begin{theorem}[Schrijver]\label{thm:inddel}
For any graph $G$, 
\[\alpha(G) \le \vartheta'(G) \le \vartheta(G).\]
\end{theorem}

Recall that $G(n, k, 1)$ is the graph whose edge-set consists of all $2$-sets $\{A, B\}$ with $|A \cap B| = 1$. We will show that $\vartheta'(G(n, k, 1)) \le \binom{n-2}{k-2}$ when $3k-3\le n\le k^2-k+1$. We shall define a matrix $M$ which satisfies the properties in the definition of $\vartheta'$ for the graph $G(n, k, 1)$ and which has largest eigenvalue $\binom{n-2}{k-2}$. In fact, we define $M$ by using Wilson's matrix~\cite{Wilson1984} from the proof of the $t$-intersecting Erd\H{o}s-Ko-Rado theorem.

Let us recall Wilson's matrix. Define

\begin{equation}\label{wilsonmatrix}
A_t = \sum_{i=0}^{t-1}(-1)^{t-1-i}\binom{k-1-i}{k-t}\binom{n-k-t+i}{k-t}^{-1}D_{k-i},
\end{equation}
where if $\alpha$ and $\beta$ are two $k$-subsets, then
\[(D_i)_{\alpha, \beta} = \binom{|\alpha \setminus \beta|}{i} = \binom{k-|\alpha \cap \beta|}{i}.\] 

Wilson~\cite[Lemma II]{Wilson1984} proved the following properties of the eigenvalues of the matrix $A_t$, in order to prove the $t$-intersecting Erd\H{o}s-Ko-Rado theorem. 

\begin{lemma}\label{lem:wilsoneigs}
Let $\theta_0, \theta_1, \ldots,  \theta_k$ be the eigenvalues of the matrix $A_t$ defined in \eqref{wilsonmatrix}. Then, 
\begin{enumerate}
\item $\theta_0 = \binom{n}{k}\binom{n-t}{k-t}^{-1}-1.$
\item $\theta_1 =  \cdots = \theta_t  = -1$. 
\item Assume $n\ge (t+1)(k-t+1)$. Then, $\theta_i \ge -1$ for $i=1, \ldots, k$. 
\end{enumerate}
\end{lemma}

We use Lemma~\ref{lem:wilsoneigs} to define a matrix $M$ which fits the definition of $\vartheta'$ and which has the appropriate largest eigenvalue. 

\begin{theorem}\label{thm:wilsonworks}
Let $M$ be the $\binom{n}{k}\times \binom{n}{k}$ matrix defined by $M = J - \binom{n-2}{k-2}A_2$, where $J$ is the $\binom{n}{k} \times \binom{n}{k}$ all-$1$s matrix and $A_2$ is the matrix defined in \eqref{wilsonmatrix} with $t=2$. Then, 
\begin{enumerate}
\item Let $A$ and $B$ be two $k$-element subsets of $[n]$ with $|A\cap B| \neq 1$. For $n\le k^2-k+1$, we have that \[M_{A, B} \ge 1.\]
\item For $n\ge 3k-3$, the largest eigenvalue of $M$ is $\binom{n-2}{k-2}$. 
\end{enumerate}
\end{theorem}

\begin{proof}[Proof of Theorem~\ref{thm:wilsonworks}]
Item 2 of the theorem statement is a direct consequence of Lemma~\ref{lem:wilsoneigs}. We now prove the first item in the theorem. 
If $A$ and $B$ are two sets with $|A\cap B|\ge 2$, then $M_{A, B} = J_{A, B} = 1$. We therefore suppose that $|A\cap B| = 0$. Then the $(A, B)$-entry of $M$ is 
\[1 - \binom{n-2}{k-2}\left((-1)\frac{(k-1)}{\binom{n-k-2}{k-2}} + \frac{k}{\binom{n-k-1}{k-2}}\right).\]
We show that this expression is at least $1$ when $n\le k^2-k+1$. It suffices to show 
\[(-1)\frac{(k-1)}{\binom{n-k-2}{k-2}} + \frac{k}{\binom{n-k-1}{k-2}} \le 0.\]
This inequality is equivalent to 
\[\frac{k\binom{n-k-2}{k-2}}{\binom{n-k-1}{k-2}} \le k-1,\]
which after simplifying the binomial coefficients is equivalent to 
\[\frac{n-k-1}{n-2k+1} \ge \frac{k}{k-1} \iff n\le k^2-k+1.\]
\end{proof}

Theorem~\ref{thm:wilsonworks} quickly implies Theorem~\ref{thm:nosingletonint}. 

\begin{proof}[Proof of Theorem~\ref{thm:nosingletonint}]
The matrix $M$ fits in the definition of $\vartheta'$, so Theorem~\ref{thm:wilsonworks} implies that $\vartheta'(G(n, k, 1)) = \binom{n-2}{k-2}$ when $3k-3\le n \le k^2-k+1$. It follows by Theorem~\ref{thm:inddel} that $\alpha(G(n, k, 1)) = \binom{n-2}{k-2}$. 
\end{proof}

If $n=k^2-k+1$, then the calculation in the proof of Theorem~\ref{thm:wilsonworks} shows $M_{A, B} = 1$ if $|A\cap B| = 0$, so the matrix $M$ also satisfies the condition in the definition of the Lov\'asz number, implying $\vartheta(G(k^2-k+1, k, 1)) = \binom{k^2-k-1}{k-2}$, which is exactly the result proved by Cherkashin. Thus, the Shannon capacity of the graph $G(k^2-k+1, k, 1)$ is also equal to $\binom{k^2-k-1}{k-2}$. 

We show that $\vartheta(G(n, k, 1)) > \binom{n-2}{k-2}$ when $n < k^2-k+1$. In particular, the graphs $G(n, k, 1)$ with $3k-3\le n < k^2-k+1$ and $k\ge 3$ provide an infinite family of graphs with $\vartheta'(G) < \vartheta(G)$. (Schrijver~\cite{Schrijver1979} gives an example of a graph $G$ with $\vartheta'(G) < \vartheta(G)$ due to M.R. Best at the end of his paper). 

\begin{theorem}\label{thm:gkk-11lovasztheta}
If $3k-3\le n < k^2-k+1$ and $k\ge 3$,  then 
\[\vartheta(G(n, k, 1)) > \binom{n-2}{k-2}.\]
\end{theorem}

\begin{proof}[Proof of Theorem~\ref{thm:gkk-11lovasztheta}]
The graphs $G(n, k, 1)$ are edge-transitive, so by a theorem of Lov\'asz~\cite[Theorem 9]{Lovasz1979}, 
\[\vartheta(G(n, k, 1)) = \frac{-\lambda_{\binom{n}{k}}}{\lambda_1 - \lambda_{\binom{n}{k}}}\binom{n}{k},\]
where $\lambda_1$ is the maximum eigenvalue of $G(n, k, 1)$ and $\lambda_{\binom{n}{k}}$ is the minimum eigenvalue. 
The eigenvalues of the graph $G(n, k, 1)$ are as follows (see, for example, \cite{Godsil2016}). 

\begin{lemma}[Eigenvalues of $G(n, k, 1)$]\label{genknesereigs}
The eigenvalues of $G(n, k, 1)$ are 
\begin{align*}
p_{k-1}(j) &= \sum_{r=k-1}^{k}(-1)^{r-k+\ell+j}\binom{r}{k-1}\binom{n-2r}{k-r}\binom{n-r-j}{r-j}\\
		 & = \sum_{r=0}^{k-1}(-1)^r\binom{j}{r}\binom{k-j}{k-1-r}\binom{n-k-j}{k-1-r}
\end{align*}
for $j = 0, \ldots, k$.
\end{lemma}

Brouwer, Cioab\u{a}, Ihringer, and McGinnis~\cite[Theorem 3.10]{Brouwer2018} showed that the minimum eigenvalue of $G(n, k, 1)$ is attained at $j=1$ if $n\le k^2-k+1$. Hence, we have $\lambda_{\binom{n}{k}} = p_{k-1}(1) = \binom{n-k-1}{k-1}-(k-1)\binom{n-k-1}{k-2} = \frac{n-k^2}{n-2k+1}\binom{n-k-1}{k-1}$, while $\lambda_1 = p_{k-1}(0) = k\binom{n-k}{k-1}$. Therefore, 
\begin{align*}
\vartheta(G(n, k, 1)) &= \frac{\frac{k^2-n}{n-2k+1}\binom{n-k-1}{k-1}}{k\binom{n-k}{k-1}+ \frac{k^2-n}{n-2k+1}\binom{n-k-1}{k-1}}\binom{n}{k}\\
&= \frac{\frac{k^2-n}{n-2k+1}\binom{n-k-1}{k-1}}{\frac{k(n-k)}{n-2k+1}\binom{n-k-1}{k-1}+ \frac{k^2-n}{n-2k+1}\binom{n-k-1}{k-1}}\binom{n}{k}\\
&=\frac{k^2-n}{n(k-1)}\binom{n}{k}\\
\end{align*}
To complete the proof, we need to check that $\frac{k^2-n}{n(k-1)}\binom{n}{k} > \binom{n-2}{k-2}$ when $3k-3 \le n < k^2-k+1$. Since $\binom{n-2}{k-2} = \frac{k(k-1)}{n(n-1)}\binom{n}{k}$, it suffices to prove
\[\frac{k^2-n}{n(k-1)} > \frac{k(k-1)}{n(n-1)} \iff -(n-k)(n-(k^2-k+1)) > 0,\]
which clearly holds if $3k-3 \le n < k^2-k+1$.  
\end{proof}

Finally, let us note that neither the Lov\'asz number nor the Delsarte number can be used to determine $\alpha(G(n, k, 1))$ for large $n$.  Indeed, Theorem~\ref{thm:franklsingleton} implies that $\alpha(G(n, k, 1)) = \binom{n-2}{k-2}$ for large $n$, but it follows from a result of the author~\cite{Linz2025} that as $n\rightarrow \infty$ 
\[\vartheta(G(n, k, 1)) = \Theta(n^{k-1})\]
and also 
\[\vartheta'(G(n, k, 1)) = \Theta(n^{k-1}).\]

\section{Concluding remarks} 

The method used to prove Theorem~\ref{thm:nosingletonint} can also be used for other $L$-systems for small values of $n$. One example of such a theorem is the following. 

\begin{theorem}\label{thm:t-2}
Let $t$ and $k$ be integers with $2\le t\le \frac{k}{2} + 1$. Let $\cF \subset \binom{[n]}{k}$ be a $(n, k, \{t-2\} \cup \{t, \ldots,  k-1\})$-system. Then for $(t+1)(k-t+1)\le n\le  k^2+k(3-2t)+(t-1)^2$, we have 
\[|\cF| \le \binom{n-t}{k-t}.\]
\end{theorem}

The proof of Theorem~\ref{thm:t-2} is completely analogous to the proof of Theorem~\ref{thm:nosingletonint}. In particular, it can be shown that for $L = \{t-2\} \cup \{t, \ldots, k-1\}$, the graph $G(n, k, L)$ has Delsarte number $\vartheta'(G(n, k, L)) = \binom{n-t}{k-t}$. Furthermore, if $n=k^2+k(3-2t)+(t-1)^2$, then $\alpha(G(n, k, L)) = \vartheta'(G(n, k, L)) = \vartheta(G(n, k, L)) = \theta(G(n, k, L))$.



\begin{thebibliography}{00}

\bibitem{Ahlswede1997}
R.~Ahlswede, L.~Khachatrian, The Complete Intersection Theorem for Systems of Finite Sets, {\it Eur. J. Combinatorics} {\bf 18} (1997), 125--136.
\bibitem{Brouwer2018}
A.E.~Brouwer, S.~Cioab\u{a}, F.~Ihringer, M.~McGinnis, The smallest eigenvalues of Hamming graphs, Johnson graphs and other distance-regular graphs with classical parameters, {\it J. Comb. Theory Ser. B} {\bf 133} (2018) 88 -- 121. 
\bibitem{Cherkashin2024}
D.~Cherkashin, On set systems without singleton intersection, {\it Discrete Math. Lett.} {\bf 14} (2024) 85--88. 
\bibitem{Delsarte1973}
P.~Delsarte, {\it An algebraic approach to the association schemes in coding theory}, Philips Res. Reps. Suppl. 10, (1973). 
\bibitem{Ellis2024}
D.~Ellis, N.~Keller, N.~Lifshitz, Stability for the complete intersection theorem, and the forbidden intersection problem of Erd\H{o}s and S\'os, {\it J. Eur. Math. Soc.} {\bf 26(5)} (2024). 
\bibitem{Erdos1975}
P.~Erd\H{o}s, Problems and results in graph theory and combinatorial analysis, In: {\it Proceedings of the Fifth British Combinatorial Conference (University Aberdeen, 1975)}, Congressus Numerantium 15, Utilitas Mathematics, Winnipeg, MB (1976), 169--192. 
\bibitem{Erdos1961}
P.~Erd\H{o}s, C.~Ko, R.~Rado, Intersection theorems for systems of finite sets, {\it Quart. J. Math. Oxford Ser. (2)}, {\bf 12} (1961) 313--320.
\bibitem{Frankl1977}
P.~Frankl, On families of finite sets no two of which intersect in a singleton, {\it Bull. Austral. Math. Soc.},  {\bf 17} (1977) 125--134. 
\bibitem{Frankl1985}
P.~Frankl, Z.~F\"{u}redi, Forbidding just one intersection, {\it J. Comb. Theory Ser. A} {\bf 39} (1985) 160 -- 176.
\bibitem{Godsil2016}
C.~Godsil, K.~Meagher, {\it Erd\H{o}s-Ko-Rado Theorems: Algebraic Approaches}, Cambridge University Press, 2016.
\bibitem{Keevash2006}
P.~Keevash, D.~Mubayi, R.~Wilson,  Set systems with no singleton intersection, {\it SIAM J. Discrete Math} {\bf 20} 4 (2006) 1031--1041. 
\bibitem{Keller2021}
N.~Keller, N.~Lifhsitz, The junta method for hypergraphs and the Erd\H{o}s-Chv\'atal simplex conjecture, {\it Adv. Math.} {\bf 392} (2021) \#107991.
\bibitem{Kupavskii2024}
A.~Kupavskii, D.~Zakharov, Spread approximations for forbidden intersection problems, {\it Adv. Math.} {\bf 445} (2024) \#109653. 
\bibitem{Linz2025}
W.~Linz, $L$-systems and the Lov\'asz number, {\it Combinatorica} {\bf 45} no.~2 (2025) Paper No. 12 24pp.
\bibitem{Lovasz1979}
L.~Lov\'asz, On the Shannon capacity of a graph, {\it IEEE Trans. Inf. Theory} {\bf IT-25}(1) (1979) 1--7.
\bibitem{Schrijver1979}
A.~Schrijver, A Comparison of the Delsarte and Lov\'asz Bounds, {\it IEEE. Trans. Inf. Theory} {\bf IT-25} (4) 425--429. 
\bibitem{Wilson1984}
R.M.~Wilson, The exact bound in the Erd\H{o}s-Ko-Rado theorem, {\it Combinatorica} {\bf 4} (1984) 247--257.
\end{thebibliography}
\end{document}